\documentclass[twoside,11pt]{amsart}
\usepackage{a4wide}
\usepackage[utf8]{inputenc}
\usepackage[english]{babel}
\usepackage[fixlanguage]{babelbib}
\usepackage{color,graphicx}
\usepackage{amssymb,amsfonts,amsmath,amsthm,amscd, epsfig}
\usepackage[all]{xy}
\usepackage{tensor}
\usepackage{microtype}
\usepackage{tikz-cd}
\usepackage{float}

\title{Globalization of Partial Actions of Ordered Groupoids on Rings}
\author[Lautenschlaeger and Tamusiunas]{Wesley G. Lautenschlaeger and Thaísa Tamusiunas$^*$}
\address{Instituto de Matem\'{a}tica, Universidade Federal do Rio Grande do Sul,  Av. Bento Gon\c{c}alves, 9500, 91509-900. Porto Alegre-RS, Brazil}
\email[Lautenschlaeger]{wesleyglautenschlaeger@gmail.com}
\email[Tamusiunas]{thaisa.tamusiunas@gmail.com}
\date{}
\thanks{$^*$ Corresponding author}

 %    \subjclass is required.

\usepackage{graphicx}
\usepackage{amssymb}
\usepackage{amsmath}
\usepackage{amsthm}
\usepackage{indentfirst}
\usepackage{tikz}
\usetikzlibrary{cd}

\newcounter{contador}
\numberwithin{contador}{section}

\newtheorem{theorem}[contador]{Theorem}
\newtheorem{prop}[contador]{Proposition}
\newtheorem{lemma}[contador]{Lemma}

\theoremstyle{definition}
\newtheorem{defi}[contador]{Definition}
\newtheorem{obs}[contador]{Remark}
\newtheorem{exe}[contador]{Example}

\newcommand{\G}{\mathcal{G}}
\newcommand{\cH}{\mathcal{H}}

\begin{document}

\begin{abstract}
   We provide a necessary and sufficient condition to the existence of an ordered globalization of a partial ordered action of an ordered groupoid on a ring, and we also present criteria to obtain uniqueness. Furthermore, we apply those results to obtain a Morita context and to show that an inverse semigroup partial action has a globalization (unique up to equivalences) if and only if it is unital.
\end{abstract}

\maketitle

\section{Introduction}

Much has been developed in the theory of actions and partial actions of groupoids. For instance, Galois theories were considered in \cite{bagio2012partial}, \cite{bagio2022galois}, \cite{cortes2017characterisation}, \cite{garta2024},
\cite{paques2013galois}, \cite{paques2018galois}, \cite{paques2021galois} and \cite{pedrotti2023injectivity}, actions in sets and duality theorems in \cite{della2023groupoid}, generalizations of classic results of group theory such as isomorphism theorems, Lagrange's Theorem, Sylow theorems, classification of groupoids and properties of cosets in \cite{avila2020notions}, 
\cite{avila2020isomorphism}, \cite{beier2023generalizations} and \cite{marin2022groupoids}, and partial actions of ordered groupoids and categories in \cite{bagio2010partial}, \cite{gilbert2005actions} and \cite{nystedt}. 

The concept of partial group action was introduced by R. Exel in \cite{exel1998partial} in order to classify a certain class of $C^*$-algebras. J. Kellendonk and M. Lawson \cite{kellaw2004} investigated applications in several other areas, such as graph immersions and inverse semigroups \cite{kellaw2004}, $\mathbb{R}$-trees \cite{Rimlinger} and tilings of the Euclidean space \cite{kellaw2000}. Much more about this theory and applications can be read in Exel's book \cite{exel2017partial}, and also in the surveys written by M. Dokuchaev in \cite{dokuchaev2011partial} and \cite{dokuchaev2019recent}, and by E. Batista in \cite{batista2017partial}, in the later especially in relation to Hopf algebras and weak Hopf algebras. 

The definition of partial action of a groupoid was given by D. Bagio and A. Paques in \cite{bagio2012partial}. Among other results, they proved a condition for the existence and uniqueness of a globalization. Partial ordered actions of ordered groupoids on rings were introduced in \cite{bagio2010partial}, but so far no study on the existence of globalization has been done.

The globalization problem has been investigated in several other contexts. For instance, in \cite{dokuchaev2005associativity} it was proved that a partial action of a group on a unital algebra is globalizable if and only if each ideal of the action is unital, and it is unique up to equivalence. For categories acting partially on sets, a globalization is not unique, but a globalization with an additional universal property, called \emph{universal globalization}, is unique up to equivalence, and such a globalization always exists \cite[Theorem 4]{nystedt}. For semigroups acting partially on sets, there exist two non-isomorphic universal globalizations \cite{kudryavtseva}. For Hopf algebras acting partially on a unital algebra, a globalization is not unique, but a globalization with an additional minimal property, called \emph{minimal globalization}, is \cite[Theorem 3]{AB}.

%developing the theory within that context, and subsequently returning to partial actions.

The problem of determining when a partial action can be realized as a restriction of a global one is very relevant, because it allows us to undestand how the partial theory behaves compared to the global one. The issue of uniqueness is particularly crucial, as it establishes a well-defined framework for transitioning to global actions. Thus, the primary aim of this work is to analyze the existence and uniqueness of ordered globalizations of partial ordered actions of ordered groupoids on rings. In Section 2, we provide a background on the partial ordered actions of ordered groupoids and establish some notation. In Section 3, we prove that a partial ordered action of an ordered groupoid on a ring has a globalization if and only if it is unital. Section 4 introduces strong partial ordered actions, pseudoassociative groupoids and minimal globalizations, showing that a unital strong partial ordered action of a pseudoassociative groupoid on a ring always has a unique minimal globalization. Finally, in Section 5, we present two applications: we construct a Morita context between the crossed products of an ordered groupoid acting partially on a ring and its globalization, and we apply our results to the case of inverse semigroups acting on rings via the Ehresmann-Schein-Nambooripad (ESN) Theorem. %(see Theorem \ref{teoesn}).

Throughout, rings and algebras are associative unless otherwise stated.

\section{Preliminaries}

We recall that a \emph{groupoid} $\G$ is a small category in which every morphism is an isomorphism. We denote by $\G_0$ the set of objects of $\G$. Observe that $\text{id} :\G_0 \rightarrow \G$, given by $\text{id}(x)=\text{id}_x$, is an injective map, from which we identify $\G_0\subseteq \G$. Given $g \in \G$, the \emph{domain} and the \emph{range} of $g$ will be denoted by $d(g)$ and $r(g)$, respectively. Hence, $d(g) = g^{-1}g$ and $r(g) = gg^{-1}$. For all $g, h \in \G$, we write $\exists gh$ whenever the product $gh$ is defined. We denote $\G_2 = \{(g,h) \in \G \times \G : \exists gh \}$.

We say that $\G$ is \emph{ordered} if there is a partial order $\leq$ on $\G$ such that: \begin{itemize}

\item [(OG1)] If $g \leq h$ then $g^{-1} \leq h^{-1}$;

\item [(OG2)] For all $g,h,k,\ell \in \G$ such that $g \leq h$, $k \leq \ell$, $\exists gk$ and $\exists h\ell$, we have $gk \leq h\ell$;

\item [(OG3)] Given $g \in \G$ and $e \in \G_0$ such that $e \leq d(g)$, there is a unique element $(g | e) \in \G$ which satisfies $(g | e) \leq g$ and $d(g | e) = e$.

\item [(OG3*)] Given $g \in \G$ and $e \in \G_0$ such that $e \leq r(g)$, there is a unique element $(e | g) \in \G$ which satisfies $(e | g) \leq g$ and $r(e | g) = e$. \end{itemize}

In a partially ordered set $P$, we say that $z \in P$ is the \emph{meet} of $x \in P$ and $y \in P$ if $z \leq x$, $z \leq y$ and if $w \in P$ is such that $w \leq x,y$ then $w \leq z$. We denote $z = x \wedge y$. A poset where every pair of elements has a meet is called a \emph{meet semilattice}. If $\G$ is such that $\G_0$ is a meet semilattice with respect to $\leq$, we say that $\G$ is an \emph{inductive} groupoid.

\begin{defi}\label{par-groupoid}
Let $\G$ be a groupoid and $A$ be a ring. We define $\alpha = (A_g,\alpha_g)_{g \in \G}$ as a \emph{partial groupoid action} of $\G$ on $A$ if:

1.$A_{r(g)} \triangleleft A$, $A_g \triangleleft A_{r(g)}$, and $\alpha_g : A_{g^{-1}} \to A_g$ is an isomorphism of rings, for all $g \in \G$;

2. The following conditions hold:\begin{itemize}
    \item[(P1)] $\alpha_e = \text{Id}_{A_e}$, for all $e \in \G_0$, and $A = \sum_{e \in \G_0} A_e$;
    \item[(P2)] $\alpha_{h}^{-1}(A_{g^{-1}} \cap A_h) \subseteq A_{(gh)^{-1}}$, for all $(g,h) \in \G_2$;
    \item[(P3)] $\alpha_g \circ \alpha_h(a) = \alpha_{gh}(a)$, for all $(g,h) \in \G_2, a \in \alpha_h^{-1}(A_{g^{-1}}\cap A_h)$.
\end{itemize}\end{defi}

Furthermore, a partial action $\alpha$ is said to be a \emph{partial ordered action}, or simply \emph{P.O. action}, if the following condition holds:
\begin{enumerate}
    \item[(PO)] If $g \leq h$, then $A_g \subseteq A_h$ and $\alpha_{g} = \alpha_h|_{A_{g^{-1}}}$.
\end{enumerate}

\begin{obs} We highlight some facts concerning the definition of partial action of a groupoid. \begin{itemize}
\item[1.] The concept of a partial groupoid action was first defined by Bagio and Paques in \cite{bagio2012partial}, although the authors did not require the condition $A = \sum_{e \in \G_0} A_e$. We incorporate it into the definition in order to recover the concept of partial group actions on rings. 
\item[2.] A similar, but stronger, condition has already appeared in \cite{nystedt2018artinian}. The property (P4) of \cite{nystedt2018artinian} requires $A = \displaystyle\bigoplus_{e \in \G_0} A_e$. This type of action was later called \emph{orthogonal partial action} \cite[Section 2]{lata2021galois}. In a P.O. action of an ordered groupoid, the orthogonality is translated into $A = \displaystyle\bigoplus_{e \in \text{max }\G_0} A_e$, where $\text{max }\G_0$ is the set of maximal elements in $\G_0$ relative to the order of $G$. This means that given $f_1, f_2 \in \text{max }\G_0$ such that $A_{f_1} \neq A_{f_2}$, for each $e \in \G_0$ such that $e \leq f_1, f_2$, we have $A_e = 0$. Therefore, this condition is quite restrictive in the ordered case, and that is why we are asking for a weaker condition.
\end{itemize}
\end{obs}

By \cite[Lemma 1.1]{bagio2012partial}, it is also true that:
\begin{enumerate}
    \item[(i)] $\alpha_g^{-1} = \alpha_{g^{-1}}$, for all $g \in \G$.\label{propriedadesacpargrp}
    \item[(ii)] $\alpha_g(A_{g^{-1}} \cap A_h) = A_g \cap A_{gh}$, for all $(g,h) \in \G_2$.
\end{enumerate}

A \emph{unital partial action} of $\G$ on $A$ is a partial action $\alpha$ such that every $A_g$, for $g \in \G$, is generated by a central idempotent $1_g$ of $A$. A ring $A$  has \emph{local units} if there is a set $E \subseteq A$ of central idempotents such that for every $a \in A$ there is $b \in E$ such that $ba = a$.  We say that $\alpha$ is \emph{preunital} if $A_e$ is a unital ring with central and idempotent identity $1_e$, for all $e \in \G_0$. Notice that the condition (P1) implies that if there is a preunital partial action of $\G$ on $A$, then $A$ has local units, but it is not necessarily unital unless $\G_0$ is a finite set.

Two P.O. actions $\alpha = (A_g,\alpha_g)_{g \in \G}$ of $\G$ on a ring $A$ and $\gamma = (C_g,\gamma_g)_{g \in \G}$ of $\G$ on a ring $C$ are \emph{equivalent} (we denote by $\alpha \simeq \gamma$) if there is a collection of ring isomorphisms $\{\varphi_e : A_e \to C_e : e \in \G_0\}$ such that
\begin{enumerate}
    \item[(i)] $\varphi_{r(g)}(A_g) = C_g$, for all $g \in \G$, 

    \item[(ii)] $\varphi_{r(g)} \circ \alpha_g(a) = \gamma_g \circ \varphi_{d(g)}(a)$, for all $a \in A_{g^{-1}}$.
\end{enumerate}

A global action of $\G$ on $A$ is a partial action such that $A_g = A_{r(g)}$, for all $g \in \G$. Ordered global actions and unital global actions are defined analogously. Of course a global action is unital if and only if it is preunital.

Let $\G$ be an ordered groupoid and $\alpha = (A_g,\alpha_g)_{g \in \G}$ a P.O. action of $\G$ on a ring $A$. We define the \emph{partial skew groupoid ring} $A \ltimes_\alpha \G$ by $A \ltimes_\alpha \G = \bigoplus_{g \in \G} A_g\delta_g$, where each $\delta_g$ is a symbol. We consider the usual sum and the product is given by
\begin{align*}
    (a_g\delta_g)(b_h\delta_h) = \begin{cases}
        \alpha_{g}(\alpha_{g^{-1}}(a_g)b_h)\delta_{gh}, \text{ if } \exists gh, \\
        0, \text{ otherwise,}
    \end{cases}
\end{align*}
linearly extended. We have that $A \ltimes_\alpha \G$ is a ring, but it is not necessarily associative nor unital. A condition to ensure the associativity of  $A \ltimes_\alpha \G$ is to ask 
$\alpha$ to be a unital partial action (c.f. \cite[Proposition 3.1]{bagio2010partial} and \cite[Corollary 3.2]{dokuchaev2005associativity}). Also, by \cite[Propostion 3.3]{bagio2010partial}, if $\G_0$ is finite and $\alpha$ is preunital, then $A \ltimes_\alpha \G$ is unital with identity $1_{A \ltimes_\alpha \G} = \sum_{e \in \G_0} 1_e\delta_e$. If $\alpha$ is preunital and $\G_0$ is not finite, then $A \ltimes_\alpha \G$ is not unital in general, but it has local units, since $A$ has.

\begin{defi}
Let $\G$ be an ordered groupoid and $\alpha = (A_g,\alpha_g)_{g \in \G}$ a P.O. action of $\G$ on a ring $A$ such that $A \ltimes_\alpha \G$ is associative. Define the ideal $N$ of $A \ltimes_\alpha \G$ as $N = \langle a\delta_g - a\delta_h : g \leq h, a \in A_g \subseteq A_h \rangle$. We define the \emph{partial skew ordered groupoid ring} $A \ltimes_\alpha^o \G$ as $A \ltimes_\alpha^o \G = \frac{A \ltimes_\alpha \G}{N}$.
\end{defi}

\section{Globalization}

Let $\beta = (B_g,\beta_g)_{g \in \G}$ be an ordered action of an ordered groupoid $\G$ on a ring $B$. Let $A$ be an ideal of $B$. We shall construct a P.O. action $\alpha$ of $\G$ on $A$ by restricting $\beta$. Consider, for each $e \in \G_0$, $A_e := A \cap B_e$, which is an ideal of $A$ and $B$. Now, for all $g \in \G$, define $A_g := A_{r(g)} \cap \beta_g(A_{d(g)})$, which is an ideal of $A_{r(g)}$, and also define $\alpha_g = \beta_g|_{A_{g^{-1}}}$, which is a ring isomorphism.

Observe that if $e \leq f$ are elements of $\G_0$, then $B_e \subseteq B_f$, which implies $A_e = A \cap B_e \subseteq A \cap B_f = A_f$. Now, if $g \leq h$, then
\begin{align*}
    A_g & = A_{r(g)} \cap \beta_g(A_{d(g)}) \subseteq A_{r(h)} \cap \beta_g(A_{d(g)}) \\
    & = A_{r(h)} \cap \beta_h(A_{d(g)}) \subseteq A_{r(h)} \cap \beta_h(A_{d(h)}) = A_h,
\end{align*}
and given arbitrary $a \in A_{g^{-1}}$, $\alpha_h(a) = \beta_h(a) = \beta_g(a) = \alpha_g(a)$, so $\alpha_g = \alpha_h|_{A_{g^{-1}}}$. Hence $\alpha = (A_g,\alpha_g)_{g \in \G}$ is a P.O. action of $\G$ on $A$. We say that $\alpha$ is the \emph{standard restriction} of $\beta$ to the ideal $A$.

\begin{exe} \label{exe1glob}
Let $\G = \{s,s^{-1},r(s),d(s),e\}$ with $\G_0 = \{ r(s),d(s),e\}$, $e \leq s,$ $ and \leq s^{-1}$, $e \leq r(s)$, $e \leq d(s)$.

Let $R$ be a commutative ring and $B = Re_1 \oplus Re_2 \oplus Re_3$, where $\{e_1, e_2, e_3\}$ is a set of pairwise orthogonal idempotents whose sum is $1_B$. Define
\begin{align*}
    B_{d(s)} = B_{s^{-1}} = Re_1 \oplus Re_2, \quad B_{r(s)} = B_s = Re_2 \oplus Re_3, \quad B_e = Re_2
\end{align*}
and $\beta_g = Id_{B_g}$, for all $g \in \G_0$, $\beta_s(ae_1 + be_2) = be_2 + ae_3$, and $\beta_{s^{-1}} = \beta_s^{-1}$.

Let $A = Re_2 \oplus Re_3$. The standard restriction $\alpha = (A_g,\alpha_g)_{g \in \G}$ of $\beta$ to $A$ is given by $A_{r(s)} = Re_2 \oplus Re_3$, $A_{d(s)} = Re_2 = A_e$, $A_s = Re_2 = A_{s^{-1}}$, and $\alpha_g = \beta_g|_{A_{g^{-1}}} = Id_{A_{g^{-1}}}, \text{ for all } g \in \G$.
\end{exe}

\begin{obs}
    Every P.O. action $\alpha = (A_g,\alpha_g)_{g \in \G}$ of $\G$ on a ring $A$ obtained by a standard restriction of an ordered action $\beta = (B_g,\beta_g)_{g \in \G}$ of $\G$ on a ring $B$, where $A$ is an ideal of $B$, satisfies the following condition: if $e \leq r(g)$, then $A_{(e|g)} = A_e \cap A_g$. In fact,
\begin{align*}
    A_e \cap A_g & = A_e \cap A_{r(g)} \cap \beta_g(A_{d(g)}) = A_e \cap \beta_g(A_{d(g)}) \\
    & = \beta_g(\beta_{g^{-1}}(A_e)) \cap \beta_g(A_{d(g)}) = \beta_g(A_{d(g)} \cap \beta_{g^{-1}}(A_e)) \\
    & = \beta_g(A_{d(g)} \cap \beta_{(g^{-1}|e)}(A_e)) = \beta_g(A_{d(g)} \cap B_{r(g^{-1}|e)} \cap \beta_{(g^{-1}|e)}(A_e)) \\
    & = \beta_g(A_{d(g)} \cap B_{d(e|g)} \cap \beta_{(g^{-1}|e)}(A_e)) = \beta_g(A \cap B_{d(g)} \cap B_{d(e|g)} \cap \beta_{(g^{-1}|e)}(A_e)) \\
    & = \beta_g(A \cap B_{d(e|g)} \cap \beta_{(g^{-1}|e)}(A_e)) = \beta_g(A_{d(e|g)} \cap \beta_{(g^{-1}|e)}(A_e)) \\
    & = \beta_{(e|g)}(A_{d(e|g)} \cap \beta_{(g^{-1}|e)}(A_e)) = A_e \cap \beta_{(e|g)}(A_{d(e|g)}) = A_{(e|g)}.
\end{align*}

This fact inspires our next definition.
\end{obs}

\begin{defi} \label{defacaoparcforte}
Let $\alpha = (A_g,\alpha_g)_{g \in \G}$ be an P.O. action of $\G$ on a ring $A$. We say that $\alpha$ is a \emph{strong} P.O. action if $A_{(e|g)} = A_e \cap A_g$, for all $g \in \G$, $e \leq r(g)$.
\end{defi}

Now we present a new kind of restriction of $\beta$, which generalizes the standard one. Let $\beta = (B_g,\beta_g)_{g \in \G}$ be an ordered global action of $\G$ on a ring $B$.  Consider $A_e \subseteq B_e$ an ideal of $B$, for all $e \in \G_0$, such that $A_e \subseteq A_f$ if $e \leq f$ in $\G_0$. Then $A_g = A_{r(g)} \cap \beta_g(A_{d(g)})$ is an ideal of $A_{r(g)}$, for all $g \in \G$. Define $\alpha_g = \beta_g|_{A_{g^{-1}}}$, for all $g \in \G$. Consider $A = \sum_{e \in \G_0} A_e$, which is an ideal of $B$. Hence $\alpha = (A_g,\alpha_g)_{g \in \G}$ is a P.O. action of $\G$ on $A$. We say that $\alpha$ is a \emph{restriction} of $\beta$ to the ring $A$.

Every standard restriction of an ordered global action is a restriction. The next example illustrates that there are restrictions that are not standard.

\begin{exe} \label{exe2glob}
Let $\G = \{m_0,m_1,n_0,n_1\}$, where $\G_0 = \{m_0,n_0\}$, $m_1^2 = m_0$, $n_1^2 = n_0$, $m_0 \leq n_0$ and $m_1 \leq n_1$. Let $R$ be a commutative ring and $A = Re_1 \oplus Re_2 \oplus Re_3 \oplus Re_4$, where the $e_i$' s are pairwise orthogonal central idempotents whose sum is $1_A$. Define
\begin{align*}
    A_{m_0} = A_{n_1} = A_{n_0} = A, \qquad A_{m_1} = Re_1 \oplus Re_3,
\end{align*}
and
\begin{align*}
    \alpha_{n_1}(ae_1 + be_2 + ce_3 + de_4) & = ce_1 + de_2 + ae_3 + be_4, \\
    \alpha_{m_1}(ae_1 + ce_3) & = ce_1 + ae_3, \\
    \alpha_{e} = I_{A_e}, & \text{ for all } e \in \G_0,
\end{align*}
for all $a,b,c,d \in R$.

Then $\alpha = (A_g,\alpha_g)_{g \in \G}$ is a unital P.O. action of $\G$ on $A$. In Theorem \ref{teoglob} we will see that $\alpha$ is a restriction of some ordered action $\beta$ of $\G$ on some ring $B$. However, $\alpha$ is not a standard restriction, since $A_{(m_0|n_1)} = A_{m_1} \neq A_{n_1} = A_{m_0} \cap A_{n_1}$. Hence there are restrictions that are not standard.
\end{exe}

It is always possible to obtain many P.O. actions from ordered global actions using restrictions. The question we want to answer is that when we can, starting from a P.O. action $\alpha$, give rise to a global ordered action $\beta$ such that a restriction of $\beta$ is equivalent to $\alpha$. We call such an ordered global action of \emph{ordered globalization}.

%\textcolor{red}{Although $C_e$ as in the construction above is not always an ideal of $C$, the restrictions that we will construct in Theorems \ref{teoglob} and \ref{teounicglob} will always satisfy this condition, since it will coincide with the original partial action. Formally, we have the next definition.}

\begin{defi}\label{globalnotunique}
Let $\alpha = (A_g,\alpha_g)_{g \in \G}$ be a P.O. action of an ordered groupoid $\G$ on a ring $A$. An ordered global action $\beta = (B_g,\beta_g)_{g \in \G}$ of $\G$ on a ring $B$ is said to be a \emph{globalization} of $\alpha$ if for all $e \in \G_0$ there is an injective homomorphism of rings $\varphi_e : A_e \to B_e$ such that
\begin{enumerate}
    \item[(i)] $\varphi_e(A_e)$ is an ideal of $B_e$;
    
    \item[(ii)] $\varphi_{r(g)}(A_g) = \varphi_{r(g)}(A_{r(g)}) \cap \beta_g(\varphi_{d(g)}(A_{d(g)}))$, for all $g \in \G$;
    
    \item[(iii)] $\beta_g \circ \varphi_{d(g)}(a) = \varphi_{r(g)}  \circ \alpha_g(a)$, for all $a \in A_{g^{-1}}$ and $g \in \G$;
    
    \item[(iv)] $B_g = \sum_{r(h) \leq r(g)} \beta_h\left ( \varphi_{d(h)}(A_{d(h)}) \right)$.
\end{enumerate}
\end{defi}

% The reader should notice that the condition (iv) is not usual in the context of globalizations, (e.g. \cite{bagio2012partial}), but it is precisely what we will use to guarantee the existence of an ordered globalization on our next result. Indeed, 

Next, we shall prove that a preunital P.O. action has an ordered globalization if and only if it is unital.

\begin{theorem} \label{teoglob}
Let $\alpha = (A_g,\alpha_g)_{g \in \G}$ be a preunital P.O. action of an ordered groupoid $\G$ on a ring $A$. Then $\alpha$ has a globalization $\beta$ if and only if $\alpha$ is unital.
\end{theorem}
\begin{proof}
($\Rightarrow$): Once $\varphi_{r(g)}(A_g) = \varphi_{r(g)}(A_{r(g)}) \cap \beta_g(\varphi_{d(g)}(A_{d(g)}))$ and $A_{r(g)}$ and $A_{d(g)}$ are unital, the assertion clearly holds.

($\Leftarrow$): Assume that every $A_g$ is unital with identity $1_g$. Let $\mathcal{F} := \mathcal{F}(\G,A)$ be the ring of all maps $\G \to A$ and define $\G_g = \{ h \in \G : r(h) \leq r(g) \}$.

Take $F_g = \{ f \in \mathcal{F} : f(h) = 0, \text{ for all } h \notin \G_g \}$. Since $\G_g = \G_{r(g)}$, it follows that $F_g = F_{r(g)}$. If $h \leq g$, then $r(h) \leq r(g)$, so $F_h \subseteq F_g$. Clearly $F_g$ is an ideal of $\mathcal{F}$, for all $g \in \G$.

Now fix the notation $f(h) := f|_h$, for all $f \in \mathcal{F}$ and $h \in \G$.

For $g \in \G$ and $f \in F_{g^{-1}}$, define $\gamma_g : F_{g^{-1}} \to F_g$ by
\begin{align*}
    \gamma_g(f)|_h = \begin{cases}
        f((g^{-1}|r(h))h), \text{ if } h \in \G_g, \\
        0, \text{ otherwise.}
    \end{cases}
\end{align*}

Notice that if $h \in \G_g$, then $r(h) \leq r(g) = d(g^{-1})$, and hence the restriction $(g^{-1}|r(h))$ is defined, as well as the product $(g^{-1}|r(h))h$. Moreover, $r((g^{-1}|r(h))h) = r(g^{-1}|r(h)) \leq r(g^{-1})$, from where it follows that $\gamma_g$ is well defined. It is easy to see that $\gamma_g$ is a ring homomorphism.

For $g \in \G$, $f \in F_{g^{-1}}$ and $h \in \G_{g^{-1}}$, we have that
\begin{align*}
    \gamma_{g^{-1}} \circ \gamma_g(f)|_h & = \gamma_g(f)|_{(g|r(h))h} = f((g^{-1}|r(g|r(h)))(g|r(h))h) \\
    & \stackrel{(*)}{=} f((g^{-1}g|r(h))h) = f((d(g)|r(h))h) \stackrel{(**)}{=} f(r(h)h) = f(h),
\end{align*}
where the equality ($*$) follows by \cite[Proposition 4.1.3]{lawson1998inverse} and the equality ($**$) follows since $r(h) \leq d(g)$. Similarly $\gamma_g \circ \gamma_{g^{-1}}(f)|_h = f(h)$. Thus we have that $\gamma_g$ is a ring isomorphism.

Observe that $\gamma_e = I_{F_e}$, for all $e \in \G_0$ and
\begin{align*}
    \gamma_{gh}(f)|_k & = f((h^{-1}g^{-1}|r(k))k) \stackrel{(*)}{=} f((h^{-1}|r(g^{-1}|r(k)))(g^{-1}|r(k))k) = \gamma_g \circ \gamma_h(f)|_k,
\end{align*}
for all $f \in F_{h^{-1}}$, $k \in \G_g$ and $(g,h) \in \G_2$, where the equality ($*$) follows again by \cite[Proposition 4.1.3]{lawson1998inverse}. Then $\gamma = (F_g, \gamma_g)_{g \in \G}$ is a global action of $\G$ on $\mathcal{F}$.

To see that $\gamma$ is ordered, observe that if $h \leq g$ in $\G$ and $f \in F_{h^{-1}}$, then
\begin{align*}
    \gamma_g(f)|_k & = f((g^{-1}|r(k))k) \stackrel{(*)}{=} f((h^{-1}|r(k))k) = \gamma_h(f)|_k,
\end{align*}
that is, $\gamma_h = \gamma_g|_{F_{h^{-1}}}$. The equality $(*)$ follows from the uniqueness of restrictions, because $(h^{-1}|r(k)) \leq h^{-1} \leq g^{-1}$ and $d(h^{-1}|r(k)) = r(k)$. Hence $\gamma$ is an ordered global action of $\G$ on $\mathcal{F}$.

Now, for all $e \in \G_0$, define $\varphi_e : A_e \to F_e$ by
\begin{align*}
    \varphi_e(a)|_h = \begin{cases}
        \alpha_{h^{-1}}(a1_h), \text{ if } r(h) = e, \\
        0, \text{ otherwise,}
    \end{cases}
\end{align*}
for all $a \in A_{e}, h \in \G$. We have that $\varphi_e(a)|_e = \alpha_e(a1_e) = a$, thus $\varphi_e$ is a monomorphism, for all $e \in \G_0$. 

Let $B_g$ be the subring generated by $\bigcup_{r(h) \leq r(g)} \gamma_h(\varphi_{d(h)}(A_{d(h)}))$, for all $g \in \G$. Notice that $B_g \subseteq F_g$. Consider $B = \sum_{e \in \G_0} B_e$. Define $\beta_g := \gamma_g|_{B_{g^{-1}}}$, for all $g \in \G$. Then $\beta = (B_g,\beta_g)_{g \in \G}$ is an ordered global action of $\G$ on $B$. Observe that $\beta$ is not a unital action in general. The next step is to show that $\beta$ is a globalization of $\alpha$.

We will start by verifying (iii). Let $g \in \G$, $a \in A_{g^{-1}}$ and $h \in \G$. If $r(h) = r(g)$, then
\begin{align*}
    \varphi_{r(g)}(\alpha_g(a))|_h = \alpha_{h^{-1}}(\alpha_g(a)1_h)
\end{align*}
and
\begin{align*}
    \beta_g(\varphi_{d(g)}(a))|_h = \varphi_{d(g)}(a)|_{g^{-1}h} = \alpha_{h^{-1}g}(a1_{g^{-1}h}).
\end{align*}

Since $a \in A_{g^{-1}}$, $a1_{g^{-1}h} \in A_{g^{-1}} \cap A_{g^{-1}h}$. Thus
\begin{align*}
    \beta_g(\varphi_{d(g)}(a)) & = \alpha_{h^{-1}g}(a1_{g^{-1}h}) = \alpha_{h^{-1}}(\alpha_g(a1_{g^{-1}h}1_{g^{-1}})) \\
    & = \alpha_{h^{-1}}(\alpha_g(a1_{g^{-1}})\alpha_g(1_{g^{-1}h}1_{g^{-1}})) = \alpha_{h^{-1}}(\alpha_g(a1_{g^{-1}})1_g1_h) \\
    & = \alpha_{h^{-1}}(\alpha_g(a1_{g^{-1}})1_h),
\end{align*}
because $\alpha_g(A_{g^{-1}} \cap A_{g^{-1}h}) = A_g \cap A_h$ implies $\alpha_g(1_{g^{-1}h}1_{g^{-1}}) = 1_g1_h$.

On the one hand, if $r(g) \neq r(h)$, we have that $\varphi_{r(g)}(\alpha_g(a))|_h = 0$. On the other hand, if $r(h) < r(g)$,
\begin{align*}
    \beta_g(\varphi_{d(g)}(a))|_h = \varphi_{d(g)}(a)|_{(g^{-1}|r(h))h} = \begin{cases}
        \alpha_{h^{-1}(r(h)|g)}(a1_{(g^{-1}|r(h))h}), \text{ if } r(g^{-1}|r(h)) = d(g), \\
        0, \text{ otherwise.}
    \end{cases}
\end{align*}

Notice that if $r(g^{-1}|r(h)) = d(g)$, then since $g^{-1} \leq g^{-1}$ and $r(g^{-1}) = d(g)$, from the uniqueness of corestrictions, we have that $g^{-1} = (d(g)|g^{-1}) = (d(g)|(g^{-1}|r(h))) = (g^{-1}|r(h))$. But this implies that $r(g) = d(g^{-1}) = d(g^{-1}|r(h)) = r(h)$, which is a contradiction. Therefore $d(g^{-1}|r(h)) < d(g^{-1})$. Thus $\beta_g(\varphi_{d(g)}(a))|_h = 0$. If $r(g) < r(h)$ or $r(g)$ and $r(h)$ are incomparable, this equality  also holds, showing (iii).

Now we proceed to (ii). Let $g \in \G$ and $c \in \varphi_{r(g)}(A_{r(g)}) \cap \beta_g(\varphi_{d(g)}(A_{d(g)}))$. Then there are $a \in A_{r(g)}$, $b \in A_{d(g)}$ such that $c = \varphi_{r(g)}(a) = \beta_g(\varphi_{d(g)}(b))$. Therefore,
\begin{align*}
    a = \alpha_{r(g)^{-1}}(a1_{r(g)}) = \varphi_{r(g)}(a)|_{r(g)} = \beta_g(\varphi_{d(g)}(b))|_{r(g)} = \varphi_{d(g)}(b)|_{g^{-1}} = \alpha_g(b1_{g^{-1}}) \in A_g,
\end{align*}
from where $c = \varphi_{r(g)}(a) \in \varphi_{r(g)}(A_g)$.

Conversely, if $c \in \varphi_{r(g)}(A_g)$, then $c = \varphi_{r(g)}(a)$, for some $a \in A_g$. Taking $b = \alpha_{g^{-1}}(a) \in A_{g^{-1}}$, we have by (iii) that $\beta_g(\varphi_{d(g)}(b)) = \varphi_{r(g)}(\alpha_g(b)) = \varphi_{r(g)}(a)$. Hence $c \in \varphi_{r(g)}(A_{g}) \cap \beta_g(\varphi_{d(g)}(A_{g^{-1}})) \subseteq \varphi_{r(g)}(A_{r(g)}) \cap \beta_g(\varphi_{d(g)}(A_{d(g)}))$.

It only remains for us to show that $\varphi_e(A_e)$ is an ideal of $B_e$, for all $e \in \G_0$. For this, it is sufficient to show that $\varphi_e(b)\beta_h(\varphi_{d(h)}(a)), \beta_h(\varphi_{d(h)}(a))\varphi_e(b) \in \varphi_e(A_e)$, for all $h \in \G_e, a \in A_{d(h)}$ and $b \in A_e$. Given $k \in \G$, we have two possibilities:

\noindent \textbf{Case (1):} $r(k) \leq r(h) \leq e$.

Observe that
\begin{align*}
    \beta_h(\varphi_{d(h)}(a))\varphi_e(b)|_k = \beta_h(\varphi_{d(h)}(a))|_k\varphi_e(b)|_k = \varphi_{d(h)}(a)|_{(h^{-1}|r(k))k}\varphi_e(b)|_k.
\end{align*}

We have that
\begin{align*}
    \varphi_e(b)|_k = \begin{cases}
        \alpha_{k^{-1}}(b1_k), \text{ if } r(k) = e, \\
        0, \text{ otherwise.}
    \end{cases}
\end{align*}

Since $r(k) = e$ and $r(k) \leq r(h) \leq e$ imply $r(k) = r(h) = e$, we obtain
\begin{align*}
    \beta_h(\varphi_{d(h)}(a))\varphi_e(b)|_k & = \begin{cases}
        \alpha_{k^{-1}h}(a1_{h^{-1}k})\alpha_{k^{-1}}(b1_k), \text{ if } r(k) = r(h) = e, \\
        0, \text{ otherwise}
    \end{cases} \\
    & = \begin{cases}
        \alpha_{k^{-1}h}(a1_{h^{-1}k})1_{k^{-1}h}1_{k^{-1}}\alpha_{k^{-1}}(b1_k), \text{ if } r(k) = r(h) = e, \\
        0, \text{ otherwise}
    \end{cases} \\
    & \stackrel{(*)}{=} \begin{cases}
        \alpha_{k^{-1}h}(a1_{h^{-1}k})\alpha_{k^{-1}h}(1_{h^{-1}k}1_{h^{-1}})\alpha_{k^{-1}}(b1_k), \text{ if } r(k) = r(h) = e, \\
        0, \text{ otherwise}
    \end{cases} \\
    & = \begin{cases}
        \alpha_{k^{-1}h}(a1_{h^{-1}}1_{h^{-1}k})\alpha_{k^{-1}}(b1_k), \text{ if } r(k) = r(h) = e, \\
        0, \text{ otherwise}
    \end{cases} \\
    & = \begin{cases}
        \alpha_{k^{-1}}(\alpha_{h}(a1_{h^{-1}})b1_k), \text{ if } r(k) = r(h) = e, \\
        0, \text{ otherwise}
    \end{cases} \\
    & = \varphi_e(\alpha_h(a1_{h^{-1}})b)|_k,
\end{align*}
where $(*)$ holds because Proposition \ref{propriedadesacpargrp}(ii) implies that $\alpha_{k^{-1}h}(A_{h^{-1}k} \cap A_{h^{-1}}) = A_{k^{-1}h} \cap A_{k^{-1}}$.

\noindent \textbf{Case (2):} $r(h) \leq r(k) \leq e$ or $r(k)$ and $r(h)$ incomparable.

In this case $\beta_h(\varphi_{d(h)}(a))\varphi_e(b)|_k = \beta_h(\varphi_{d(h)}(a))|_k\varphi_e(b)|_k = 0 \cdot \varphi_e(b)|_k = 0$ directly from the definition of $\beta$.

The verification that $\varphi_e(b)\beta_h(\varphi_{d(h)}(a)) \in \varphi_e(A_e)$ is similar, and then $\varphi_e(A_e)$ is an ideal of $B_e$. 
\end{proof}

\begin{obs}
    \begin{itemize}
        \item[(i)] The ring $B$ constructed above does not need to have an identity element. Ferrero and Lazzarin have determined in \cite{ferrero2008partial} necessary and sufficient conditions on the partial group action $\alpha$ for the ring $B$ to be unital. In this scenario, they called $\alpha$ as \emph{finite type}. 

        \item[(ii)] Even though the ring $B$ is not unital, it has local units by construction, and hence it is an idempotent ring, that is, $B^2 = B$.
    \end{itemize}
\end{obs}

\begin{exe} \label{exsemunicidade}
Let $\G$, $R$, $B$ and $A$ be as in Example \ref{exe1glob}. We will construct a globalization for $\alpha$. Consider $\mathcal{F} = \mathcal{F}(\G,A)$. We have that $\mathcal{F} \simeq A^5$ via $f \in \mathcal{F} \mapsto (f(s),f(s^{-1}),f(r(s)),f(d(s)),f(e))$.

Since $\G_s = \G_{r(s)} = \{ s, r(s), e\}$,  $\G_{s^{-1}} = \G_{d(s)} = \{ s^{-1}, d(s), e\}$ and $\G_e = \{ e \}$, we have that
\begin{align*}
    F_s & = \{ f \in \mathcal{F} : f(g) = 0, \text{ for all } g \notin \G_s \} \simeq A \times 0 \times A \times 0 \times A, \\
    F_{s^{-1}} & = \{ f \in \mathcal{F} : f(g) = 0, \text{ for all } g \notin \G_{s^{-1}} \} \simeq 0 \times A \times 0 \times A \times A, \\
    F_e & = \{ f \in \mathcal{F} : f(g) = 0, \text{ for all } g \notin \G_e \} \simeq 0 \times 0 \times 0 \times 0 \times A.
\end{align*}

Thus, $\gamma_s(0,a,0,b,c) = (b,0,a,0,c)$, $\gamma_{s^{-1}}(a,0,b,0,c) = (0,b,0,a,c)$ and $\gamma_g = Id_{F_g}$, for all $g \in \G_0$. Furthermore, notice that $\varphi_{r(s)}(ae_2 + be_3) = (ae_2,0,ae_2 + be_3, 0, 0)$, $\varphi_{d(s)}(ae_2) = (0,ae_2,0,ae_2, 0)$ and $\varphi_e(ae_2) = (0,0,0,0,ae_2)$.

Given $ae_2 + be_3 \in A_{r(s)}$, $ce_2 \in A_{d(s)}$, $de_2 \in A_e$, we obtain
\begin{align*}
    & \gamma_{r(s)}(\varphi_{r(s)}(ae_2 + be_3)) + \gamma_s(\varphi_{d(s)}(ce_2)) + \gamma_e(\varphi_e(de_2)) \\
    = \; & \gamma_{r(s)}(ae_2,0,ae_2+be_3,0,0) + \gamma_s(0,ce_2,0,ce_2,0) + \gamma_e(0,0,0,0,de_2) \\
    = \; & (ae_2,0,ae_2+be_3,0,0) + (ce_2,0,ce_2,0,0) + (0,0,0,0,de_2) \\
    = \; & ((a+c)e_2,0,(a+c)e_2+be_3,0,de_2),
\end{align*}
from where it follows that $B_s' = \{ (ae_2,0,ae_2 + be_3,0,ce_2) : a,b,c \in R \}$.

Similarly, $B_{s^{-1}}' = \{ (0,ae_2+be_3,0,ae_2,0,ce_2) : a,b,c \in R \}$ and $B_e' = Re_2(0,0,0,0,1)$, thus $B' = B_s' + B_{s^{-1}}' + B_e' = \{(ae_2, be_2+ce_3,ae_2+de_3,be_2,fe_3) : a,b,c,d,f \in R\}$.

Then $\beta' = (B_g',\beta_g')_{g \in \G}$ is a globalization of $\alpha$. Observe that $\beta$ is also a globalization of $\alpha$, taking $\varphi_g = \iota_g : A_g \to B_g$ as the usual inclusion $A_g \mapsto A \cap B_g$, for all $g \in \G_0$. However $\beta' \not\simeq \beta$, since $B_s \simeq R^2 \not\simeq R^3 \simeq B_s'$.
\end{exe}

\section{Uniqueness}

In this section, we provide a sufficient condition for a P.O. action to have a unique globalization. In order to achieve this goal, we introduce the notion of a minimal globalization.

\begin{defi}
    Let $\alpha = (A_g,\alpha_g)_{g \in \G}$ be a P.O. action of $\G$ on a ring $A$ and $\beta = (B_g,\beta_g)_{g \in \G}$ be a globalization of $\alpha$ on a ring $B$. We say that $\beta$ is a \emph{minimal globalization} if 
    \begin{enumerate}
        \item[(iv')] $B_g = \sum_{r(h) = r(g)} \beta_h(\varphi_{d(h)}(A_{d(h)}))$, for all $g \in \G$,
    \end{enumerate}
    holds.
\end{defi}

Further on in Example \ref{exeglobmin} we will see some examples of minimal ordered globalizations in contrast with ordered globalizations. Besides that, we need to study some properties of strong partial actions. 

Given an ordered groupoid $\G$, we define the \emph{pseudoproduct} $*$ as
\begin{align*}
    g * h = \begin{cases}
    (g|d(g) \wedge r(h))(d(g) \wedge r(h)|h), \text{ if } d(g) \wedge r(h) \text{ is defined}, \\
    \text{undefined, otherwise.}
    \end{cases}
\end{align*}

\begin{lemma} \label{condequivacparcialforte}
    Let $\alpha = (A_g,\alpha_g)_{g \in \G}$ be a strong P.O. action of $\G$ on a ring $A$. Then it holds
    \begin{enumerate}
        \item[(PS)] $\alpha_g \circ \alpha_h = \alpha_{g*h} \circ \text{Id}_{A_{h^{-1}}}$, for all $g,h \in \G$ such that $g * h$ is defined.
    \end{enumerate}

    Conversely, if $\alpha$ is a P.O. action that satisfies (PS), then $\alpha$ is strong.
    \end{lemma}
\begin{proof}
    Given $(g,h) \in \G_2$, we have that
    \begin{align*}
        \alpha_g \circ \alpha_h & = \alpha_g \circ \alpha_h|_{\alpha_{h^{-1}}(A_{g^{-1}} \cap A_h)} = \alpha_g \circ \alpha_h|_{A_{h^{-1}} \cap A_{(gh)^{-1}}} = \alpha_g \circ \alpha_h \circ \text{Id}_{A_{(gh)^{-1}} \cap A_{h^{-1}}}\\
        & = \alpha_{gh} \circ \text{Id}_{A_{(gh)^{-1}} \cap A_{h^{-1}}} = \alpha_{gh} \circ \text{Id}_{A_{(gh)^{-1}}} \circ \text{Id}_{A_{h^{-1}}}= \alpha_{gh} \circ \text{Id}_{A_{h^{-1}}}.
    \end{align*}

    Now, if $g,h \in \G$ are such that the pseudoproduct $g * h$ is defined, that is, such that the meet $e := d(g) \wedge r(h)$ is defined, then $\alpha_g \circ \alpha_h = \alpha_g|_{A_h \cap A_{g^{-1}}} \circ \alpha_{h}|_{\alpha_{h^{-1}}(A_h \cap A_{g^{-1}})}$.

    Observe that $A_{g^{-1}} \cap A_{h} \subseteq A_{d(g)} \cap A_{r(h)} = A_e$, since $\alpha$ is strong. Then
    \begin{align*}
        A_h \cap A_{g^{-1}} & = A_h \cap A_{g^{-1}} \cap A_e = A_h \cap A_e \cap A_{g^{-1}} \cap A_e = A_{(e|h)} \cap A_{(e|g^{-1})} = A_{(e|h)} \cap A_{(g|e)^{-1}}.
    \end{align*}

    Since $\alpha_{(g|e)} = \alpha_g|_{A_{(g|e)^{-1}}}$ and $\alpha_{(e|h)} = \alpha_h|_{A_{(e|h)^{-1}}}$ we obtain
    \begin{align*}
        \alpha_g \circ \alpha_h & = \alpha_g|_{A_{(e|h)} \cap A_{(g|e)^{-1}}} \circ \alpha_h|_{\alpha_{(e|h)^{-1}}(A_{(e|h)} \cap A_{(g|e)^{-1}})} 
        \\ & = \alpha_{(g|e)}|_{A_{(e|h)} \cap A_{(g|e)^{-1}}} \circ \alpha_{(e|h)}|_{\alpha_{(e|h)^{-1}}(A_{(e|h)} \cap A_{(g|e)^{-1}})} = \alpha_{(g|e)} \circ \alpha_{(e|h)} \\
        & = \alpha_{(g|e)(e|h)} \circ \text{Id}_{A_{(e|h)^{-1}}} = \alpha_{g * h} \circ \text{Id}_{A_{(e|h)^{-1}}} = \alpha_{g * h} \circ \text{Id}_{A_{(h^{-1}|e)}}.
    \end{align*}
    
    By the uniqueness of corestriction we have that $(h^{-1}|e) = (r(h^{-1}|e)|h^{-1})$. In this way, $A_{(h^{-1}|e)} = A_{(r(h^{-1}|e)|h^{-1})} = A_{r(h^{-1}|e)} \cap A_{h^{-1}}$, since $\alpha$ is strong. Now, $r(h^{-1}|e) = r((e|h)^{-1}) = d(e|h)$. Hnece $A_{(h^{-1}|e)} = A_{d(e|h)} \cap A_{h^{-1}}$. However $d(g*h) = d(e|h)$, from where $A_{(g*h)^{-1}} \subseteq A_{d(e|h)}$. Therefore
    \begin{align*}
        \alpha_g \circ \alpha_h & = \alpha_{g * h} \circ \text{Id}_{A_{d(e|h)}} \circ \text{Id}_{A_{h^{-1}}} = \alpha_{g * h} \circ \text{Id}_{A_{(g*h)^{-1}}} \circ \text{Id}_{A_{d(e|h)}} \circ \text{Id}_{A_{h^{-1}}} \\
        & = \alpha_{g * h} \circ \text{Id}_{A_{(g*h)^{-1}}} \circ \text{Id}_{A_{h^{-1}}} = \alpha_{g * h} \circ \text{Id}_{A_{h^{-1}}},
    \end{align*}
    which is precisely (PS).

    For the converse, it only takes us to notice that if $e \leq r(g)$, then
    \begin{align*}
        \alpha_e \circ \alpha_g & = \alpha_{e * g} \circ \text{Id}_{A_{g^{-1}}} = \alpha_{(e|g)} \circ \text{Id}_{A_{g^{-1}}} = \alpha_{(e|g)}.
    \end{align*}

    Now, $A_{(e|g)} = \text{Im}(\alpha_{(e|g)}) = \text{Im}(\alpha_{e} \circ \alpha_g) = \alpha_e(A_e \cap A_g) = A_e \cap A_g$, showing that $\alpha$ is strong.
\end{proof}

\begin{obs}
    When $\alpha$ is strong, notice that $\text{dom}(\alpha_g \circ \alpha_h) = \text{dom}(\alpha_{g*h} \circ \text{Id}_{A_{h^{-1}}})$, for all $g,h \in \G$ such that $\exists g * h$. But this is the same as saying that $\alpha_g(A_{g^{-1}} \cap A_h) = A_g \cap A_{g*h}$.
\end{obs}

In the case where the action is global, we have that strong global ordered actions are precisely global ordered actions. In the case where $\G$ is trivially ordered with the equality, the concepts of strong P.O. action and partial action coincide. 

\begin{lemma} \label{lemaacforte}
Let $\alpha = (A_g, \alpha_g)_{g \in \G}$ be a strong P.O. action of $\G$ on a ring $A$. The following statements are true:
\begin{enumerate}
    \item[(i)] If $e \leq d(g)$, then $A_{(g|e)} = A_g \cap A_{r(g|e)}$.
    
    \item[(ii)] If $e, f \in \G_0$ are such that $e \wedge f$ is defined, then $A_{e \wedge f} = A_e \cap A_f$.
\end{enumerate}
\end{lemma}
\begin{proof}
(i): Notice that $(g|e) = (r(g|e)|g)$ by the uniqueness of restrictions. Hence $A_{(g|e)} = A_{(r(g|e)|g)} = A_g \cap A_{r(g|e)}$.

(ii): From (PS) it follows that $\alpha_e \circ \alpha_f = \alpha_{e * f} \circ \alpha_f = \alpha_{e \wedge f} \circ \alpha_f = \alpha_{e \wedge f}$. Therefore $A_{e \wedge f} = \text{Im}(\alpha_{e \wedge f}) = \text{Im}(\alpha_e \circ \alpha_f) = \alpha_e(A_e \cap A_f) = A_e \cap A_f.$ \qedhere
\end{proof}

\begin{defi}
Let $\G$ be an ordered groupoid. We say that $\G$ is \emph{pseudoassociative} if $\exists (g * h) * k \iff \exists g * (h * k)$, for all $g,h,k \in \G$.
\end{defi}

\begin{obs}
In the case that $\exists (g * h) * k$ and $\exists g * (h * k)$, we have by \cite[Lemma 4.1.6]{lawson1998inverse} that these pseudoproducts coincide. Inductive groupoids are always pseudoassociative.
\end{obs}

Now we have all the tools needed to characterize the uniqueness of minimal ordered globalizations.

\begin{theorem} \label{teounicglob}
Let $\alpha = (A_g, \alpha_g)_{g \in \G}$ be a preunital, strong P.O. action of a pseudoassociative groupoid $\G$ on a ring $A$. Then $\alpha$ admits a minimal globalization if and only if $\alpha$ is unital. Furthermore, the minimal globalization is unique up to equivalence.
\end{theorem}
\begin{proof}
Consider, as in the proof of Theorem \ref{teoglob}, $\mathcal{F} = \mathcal{F}(\G,A)$ the ring of the maps of $\G$ on $A$. Given $g \in \G$, define $\mathcal{E}_g = \{h \in \G : \exists g^{-1} * h \}$. Define $F_g = \{ f \in \mathcal{F} : f(h) = 0, \text{ for all } h \notin \mathcal{E}_g \}$.

Clearly $F_g = F_{r(g)}$, for all $g \in \G$. Consider $\gamma_g : F_{g^{-1}} \to F_g$, where 
\begin{align*}
    \gamma_g(f)|_h = \begin{cases}
            f(g^{-1} * h), \text{ if } h \in \mathcal{E}_g, \\
            0, \text{ otherwise.}
        \end{cases}
\end{align*}

It is easy to see that $\gamma_g$ is a well defined ring homomorphism, for all $g \in \G$. If $(g,h) \in \G_2$, $f \in F_{h^{-1}}$ and $k \in \mathcal{E}_g$ then
\begin{align*}
    \gamma_{gh}(f)|_k & = f((gh)^{-1}*k) = f((h^{-1} g^{-1}) * k) \\
    & = f((h^{-1} * g^{-1}) * k) = f(h^{-1} * (g^{-1} * k)) \\
    & = \gamma_h(f)|_{g^{-1} * k} = \gamma_g \circ \gamma_h(f)|_k.
\end{align*}

If $k \notin \mathcal{E}_g$, then $\gamma_{gh}(f)|_k = 0 = \gamma_g \circ \gamma_h(f)|_k$, so $\gamma_{gh} = \gamma_g \circ \gamma_h$.

For all $e \in \G_0$, define $\psi_e : A_e \to F_e$ by
\begin{align*}
    \psi_e(a)|_h = \begin{cases}
    \alpha_{h^{-1}}(a1_h), \text{ if } h \in \mathcal{E}_e, \\
    0, \text{ otherwise},
    \end{cases}
\end{align*}
for all $a \in A_e, h \in \G$. We have that $\psi_e(a)|_e = \alpha_e(a1_e) = a$, for all $a \in A_e$. Hence $\psi_e$ is a monomorphism, for all $e \in \G_0$.

Let $B_g$ be the subring of $F_g$ generated by $\bigcup_{r(h) = r(g)} \gamma_h(\psi_{d(h)}(A_{d(h)}))$, for all $g \in \G$. Notice that $B_g \subseteq F_g$ and that $B_g$ may not be unital. Consider $B = \sum_{e \in \G_0} B_e$. Define $\beta_g := \gamma_g|_{B_{g^{-1}}}$, for all $g \in \G$.

Given $e \in \G_0$, $g \in \G$ with $r(g) = e$ and $a \in A_{d(g)}$, we have that
\begin{align*}
    \beta_e \circ \beta_g(\psi_{d(g)}(a)) = \gamma_e \circ \gamma_g(\psi_{d(g)}(a)) = \gamma_{eg}(\psi_{d(g)}(a)) = \gamma_g(\psi_{d(g)}(a)) = \beta_g(\psi_{d(g)}(a)),
\end{align*}
from where it follows that $\beta_e = \text{Id}_{B_e}$. Since $\gamma_g \circ \gamma_h = \gamma_{gh}$, for all $(g,h) \in \G_2$, then $\beta_g \circ \beta_h = \beta_{gh}$, for all $(g,h) \in \G_2$.

Now we prove that $\beta$ satisfies (PS). Indeed, given $g,h,k, \ell \in \G$ such that $h \in \mathcal{E}_g$, $r(k) = d(h)$ and $a \in A_{d(k)}$, we have that
\begin{align*}
    \beta_g \circ \beta_h(\beta_k(\psi_{d(k)}(a)))|_\ell & = \begin{cases} 
    \beta_h(\beta_k(\psi_{d(k)}(a)))|_{g^{-1} * \ell}, \text{ if } \ell \in \mathcal{E}_g, \\
    0, \text{ otherwise} 
    \end{cases} \\
    & = \begin{cases} 
    \beta_k(\psi_{d(k)}(a))|_{h^{-1} * (g^{-1} * \ell)}, \text{ if } \ell \in \mathcal{E}_g, \\
    0, \text{ otherwise}
    \end{cases} \\
    & = \begin{cases} 
    \beta_k(\psi_{d(k)}(a))|_{(h^{-1} * g^{-1}) * \ell}, \text{ if } \ell \in \mathcal{E}_g, \\
    0, \text{ otherwise}
    \end{cases} \\
    & = \begin{cases} 
    \beta_k(\psi_{d(k)}(a))|_{((d(h) * h^{-1}) * g^{-1}) * \ell}, \text{ if } \ell \in \mathcal{E}_g, \\
    0, \text{ otherwise}
    \end{cases}
    \end{align*}

    \begin{align*}
     & = \begin{cases} 
    \beta_k(\psi_{d(k)}(a))|_{d(h) * (h^{-1} * g^{-1}) * \ell}, \text{ if } \ell \in \mathcal{E}_g, \\
    0, \text{ otherwise}
    \end{cases} \\
    & = \begin{cases} 
    \beta_{d(h)}(\beta_k(\psi_{d(k)}(a)))|_{(h^{-1} * g^{-1}) * \ell}, \text{ if } \ell \in \mathcal{E}_g, \\
    0, \text{ otherwise}
    \end{cases} \\
    & =
    \beta_{g*h} \circ \beta_{d(h)}(\beta_k(\psi_{d(k)}(a)))|_{\ell}.
\end{align*}

To finish the proof, we shall show that $\beta$ satisfies (i)-(iii) from the definition of minimal globalization. 

(iii): Let $g, h \in \G$ e $a \in A_{g^{-1}}$. Then
\begin{align*}
    \psi_{r(g)} \circ \alpha_g(a)|_h & = \begin{cases}
        \alpha_{h^{-1}}(\alpha_g(a)1_h), \text{ if } h \in \mathcal{E}_g, \\
        0, \text{ otherwise}
    \end{cases}
\end{align*}
and
\begin{align*}
    \beta_g \circ \psi_{d(g)}(a)|_h & = \begin{cases}
        \psi_{d(g)}(a)|_{g^{-1} * h}, \text{ if } h \in \mathcal{E}_g, \\
        0, \text{ otherwise}
    \end{cases} = \begin{cases}
        \alpha_{h^{-1} * g}(a1_{g^{-1}*h}), \text{ if } h \in \mathcal{E}_g, \\
        0, \text{ otherwise.}
    \end{cases}
\end{align*}

Then we can work only with the case in which the meet $e := r(g) \wedge r(h)$ is defined. As $a \in A_{g^{-1}}$ we have that $a1_{g^{-1}*h} \in A_{g^{-1}} \cap A_{g^{-1}*h} = A_{(e|g^{-1})} \cap A_{g^{-1}*h}$. Hence
\begin{align*}
    \beta_g(\psi_{d(g)}(a))|_h & =  \alpha_{h^{-1} * g}(a1_{g^{-1}*h}) = \alpha_{(h^{-1}|e)} \circ \alpha_{(e|g)}(a1_{g^{-1}*h})  \\
    & = \alpha_{(h^{-1}|e)}(\alpha_{(e|g)}(a1_{(g^{-1}|e)})1_{(e|h)}) = \alpha_{h^{-1}}(\alpha_{g}(a1_{(g^{-1}|e)})1_{(e|h)}) \\
    & = \alpha_{h^{-1}}(\alpha_{g}(a)\alpha_g(1_{g^{-1}}1_{(g^{-1}|e)})1_{(e|h)}) = \alpha_{h^{-1}}(\alpha_{g}(a)\alpha_{(e|g)}(1_{(g^{-1}|e)})1_{(e|h)}) \\
    & = \alpha_{h^{-1}}(\alpha_{g}(a)1_{(e|g)}1_{(e|h)}) = \alpha_{h^{-1}}(\alpha_{g}(a)1_e1_g1_h) \\
    & = \alpha_{h^{-1}}(\alpha_{g}(a)1_{r(g)}1_{r(h)}1_g1_h)= \alpha_{h^{-1}}(\alpha_{g}(a)1_g1_h) = \alpha_{h^{-1}}(\alpha_g(a)1_h).
\end{align*}

Therefore $\beta_g \circ \psi_{d(g)}(a) = \psi_{r(g)} \circ \alpha_g(a)$.

(ii): Let $g \in \G$ and $c \in \psi_{r(g)}(A_{r(g)}) \cap \beta_g(\psi_{d(g)}(A_{d(g)}))$. Then there are $a \in A_{r(g)}, b \in A_{d(g)}$ such that $c = \psi_{r(g)}(a) = \beta_g(\psi_{d(g)}(b))$. Thus,
\begin{align*}
    a = \alpha_{r(g)^{-1}}(a1_{r(g)}) = \psi_{r(g)}(a)|_{r(g)} = \beta_g(\psi_{d(g)}(b))|_{r(g)} = \psi_{d(g)}(b)|_{g^{-1}} = \alpha_g(b1_{g^{-1}}) \in A_g,
\end{align*}
from where it follows that $c = \psi_{r(g)}(a) \in \psi_{r(g)}(A_g)$.

Conversely, if $c \in \psi_{r(g)}(A_g)$, then $c = \psi_{r(g)}(a)$, for some $a \in A_g$. Taking $b = \alpha_{g^{-1}}(a) \in A_{g^{-1}}$, we have from (iii) that $\beta_g(\psi_{d(g)}(b)) = \psi_{r(g)}(\alpha_g(b)) = \psi_{r(g)}(a) = c$.

Thus $c \in \psi_{r(g)}(A_g) \cap \beta_g(\psi_{d(g)}(A_{g^{-1}})) \subseteq \psi_{r(g)}(A_{r(g)}) \cap \beta_g(\psi_{d(g)}(A_{d(g)}))$.

(i): Let $e \in \G_0$, $h \in \G$ be such that $r(h) = e$, $a \in A_{d(h)}$ and $b \in A_e$. It is enough to prove that $\psi_e(b)\beta_h(\psi_{d(h)}(a)), \beta_h(\psi_{d(h)}(a))\psi_e(b) \in \psi_e(A_e)$. For that, take $k \in \G$. We have that $\beta_h(\psi_{d(h)}(a))\psi_e(b)|_k = \beta_h(\psi_{d(h)}(a))|_k\psi_e(b)|_k$. On the one hand,
\begin{align*}
    \psi_e(b)|_k = \begin{cases}
        \alpha_{k^{-1}}(b1_k), \text{ if } k \in \mathcal{E}_e, \\
        0, \text{ otherwise.}
    \end{cases}
\end{align*}

On the other hand,
\begin{align*}
    \beta_h(\psi_{d(h)}(a))|_k & = \begin{cases}
        \psi_{d(h)}(a)|_{h^{-1}*k}, \text{ if } k \in \mathcal{E}_h = \mathcal{E}_e, \\
        0, \text{ otherwise}
    \end{cases} = \begin{cases}
        \alpha_{k^{-1}*h}(a1_{h^{-1}*k}), \text{ if } k \in \mathcal{E}_e, \\
        0, \text{ otherwise.}
    \end{cases}
\end{align*}

Thus we can work only with the case in which $k \in \mathcal{E}_e$. Then, denoting by $f = r(k) \wedge r(h)$,
\begin{align*}
    \beta_h(\psi_{d(h)}(a))\psi_e(b)|_k & = \alpha_{k^{-1}*h}(a1_{h^{-1}*k})\alpha_{k^{-1}}(b1_{k}) = \alpha_{(k^{-1}|f)(f|h)}(a1_{(h^{-1}|f)(f|k)})\alpha_{k^{-1}}(b1_{k}) \\
    & = \alpha_{(k^{-1}|f)(f|h)}(a1_{(h^{-1}|f)(f|k)})1_{(k^{-1}|f)(f|h)}1_{k^{-1}}\alpha_{k^{-1}}(b1_{k}) \\
    & = \alpha_{(k^{-1}|f)(f|h)}(a1_{(h^{-1}|f)(f|k)})1_{(k^{-1}|f)(f|h)}1_{r(k^{-1}|f)}1_{k^{-1}}\alpha_{k^{-1}}(b1_{k}).
\end{align*}

By Lemma \ref{lemaacforte}(ii), we have that $A_{(k^{-1}|f)} = A_{k^{-1}} \cap A_{r(k^{-1}|f)}$, from where it follows that $1_{(k^{-1}|f)} = 1_{k^{-1}}1_{r(k^{-1}|f)}$. Therefore,
\begin{align*}
    \beta_h(\psi_{d(h)}(a))\psi_e(b)|_k & = \alpha_{(k^{-1}|f)(f|h)}(a1_{(h^{-1}|f)(f|k)})1_{(k^{-1}|f)(f|h)}1_{(k^{-1}|f)}\alpha_{k^{-1}}(b1_{k}).
\end{align*}

Now, by Proposition \ref{propriedadesacpargrp}(ii), we have that $$A_{(k^{-1}|f)(f|h)} \cap A_{(k^{-1}|f)} = \alpha_{(k^{-1}|f)(f|h)}(A_{(h^{-1}|f)(f|k)} \cap A_{(h^{-1}|f)}),$$ so $1_{(k^{-1}|f)(f|h)}1_{(k^{-1}|f)} = \alpha_{(k^{-1}|f)(f|h)}(1_{(h^{-1}|f)(f|k)}1_{(h^{-1}|f)})$. Thus,
\begin{align*}
    \beta_h(\psi_{d(h)}(a))\psi_e(b)|_k & = \alpha_{(k^{-1}|f)(f|h)}(a1_{(h^{-1}|f)(f|k)})\alpha_{(k^{-1}|f)(f|h)}(1_{(h^{-1}|f)(f|k)}1_{(h^{-1}|f)})\alpha_{k^{-1}}(b1_{k}) \\
    & = \alpha_{(k^{-1}|f)(f|h)}(a1_{(h^{-1}|f)(f|k)}1_{(h^{-1}|f)})\alpha_{k^{-1}}(b1_{k}) \\
    & = \alpha_{(k^{-1}|f)}(\alpha_{(f|h)}(a1_{(h^{-1}|f)})1_{(f|k)})\alpha_{k^{-1}}(b1_{k}) \\
    & = \alpha_{k^{-1}}(\alpha_{(f|h)}(a1_{(h^{-1}|f)})1_{(f|k)})\alpha_{k^{-1}}(b1_{k}) \\
    & = \alpha_{k^{-1}}(\alpha_{(f|h)}(a1_{(h^{-1}|f)})1_{(f|k)}b1_{k}) = \alpha_{k^{-1}}(\alpha_{(f|h)}(a1_{(h^{-1}|f)})1_{f}1_kb1_{k}) \\
    & = \alpha_{k^{-1}}(\alpha_{(f|h)}(a1_{(h^{-1}|f)})b1_k) = \psi_e(\alpha_{(f|h)}(a1_{(h^{-1}|f)})b)|_k.
\end{align*}

Notice that $b \in A_e$ implies $\alpha_{(f|h)}(a1_{(h^{-1}|f)})b \in A_e$, since $A_e$ is an ideal of $A$. Analogously we have that $\psi_e(b)\beta_h(\psi_{d(h)}) \in \psi_e(A_e)$.

The uniqueness now follows by \cite[Theorem 2.1]{bagio2012partial}.
\end{proof}

\begin{exe} \label{exeglobmin}
Resuming Example \ref{exsemunicidade} and applying the construction of Theorem \ref{teounicglob}, we obtain $\mathcal{E}_g = \G, \text{ for all } g \in \G$, from where $F_g = \mathcal{F}, \text{ for all } g \in \G$.

Thus, $\gamma_s(a,b,c,d,f) = (d,c,b,a,f) = \gamma_{s^{-1}}(a,b,c,d,f)$ and $\gamma_g = Id_{\mathcal{F}}$, for all $f \in \G_0$. Furthermore, notice that $\psi_{r(s)}(ae_2 + be_3) = (ae_2,ae_2,ae_2 + be_3,ae_2,ae_2)$ and $\psi_{d(s)}(ae_2) = \psi_e(ae_2) = (ae_2,ae_2,ae_2,ae_2,ae_2)$.

Therefore $B_s' = \{(ae_2,ae_2,ae_2+be_3,ae_2,ae_2) : a,b \in R \}$, $B_{s^{-1}}' = \{(ae_2, ae_2 + be_3, ae_2,$ $ ae_2, ae_2) : a,b \in R\}$ and $B_e' = Re_2(1,1,1,1,1)$. Hence $B' = B_s' + B_{s^{-1}}' + B_e' = \{(ae_2,ae_2 + be_3, ae_2 + ce_3, ae_2, ae_2) : a,b,c \in R\}$ and then $B' \simeq B$ and $B_g' \simeq B_g$, for all $g \in \G$. So $\beta' \simeq \beta$.
\end{exe}

\section{Applications}

We conclude by presenting two applications of the results of the previous sections. We first show that the partial skew ordered groupoid ring is Morita equivalent to the skew ordered groupoid ring relative to its globalization. Afterwards, as a consequence of the ESN Theorem \cite[Theorem 4.1.8]{lawson1998inverse}, we provide a necessary and sufficient condition for a partial action of an inverse semigroup to have a (unique) minimal globalization.

\subsection{A Morita Context}

To simplify the notation, in this subsection we will assume that $\varphi_e = \iota_e : A_e \to B_e$ is the usual inclusion, so that $A_e$ is an ideal of $B_e$, for all $e \in \G_0$. Also, assume that $\G_0$ is finite.

\begin{prop} \label{propmorita}
Let $\alpha = (A_g,\delta_g)_{g \in \G}$ be a unital P.O. action of the ordered groupoid $\G$ on a ring $A$ and $\beta = (B_g,\beta_g)_{g \in \G}$ be a globalization of $\alpha$ on a ring $B$. Let $R = A \ltimes_\alpha^o \G$ and $T = B \ltimes_\beta^o \G$. Then:
\begin{enumerate}
    \item[(i)] $T1_R = \sum_{g \in \G} \overline{\beta_g(A_{d(g)})\delta_g}$;
    
    \item[(ii)] $1_RT = \sum_{g \in \G} \overline{A_{r(g)}\delta_{g}}$;
    
    \item[(iii)] $1_RT1_R = R$;
    
    \item[(iv)] $T1_RT = T$.
\end{enumerate}
\end{prop}
\begin{proof}
First, notice that the crossed products $A \ltimes_\alpha \G$ and $B \ltimes_\beta \G$ are associative, so $R$ and $T$ are defined. Also, since $\alpha$ is unital and $\G_0$ is finite, we have that $1_R$ is defined. The proof of (i)-(iii) are similar to the nonordered case \cite[Proposition 3.1]{bagio2012partial}. In order to prove (iv), it suffices to show that $T \subseteq 1_RT1_R$. Since for all $h \in \G$ we have that $B_h = \sum_{r(g) \leq r(h)} \beta_g(A_{d(g)})$, the result follows from
\begin{align*}
    \overline{\beta_g(a)\delta_h} = \overline{\beta_g(a)\delta_{(r(g)|h)}} = \overline{\beta_g(a)\delta_g} \cdot \overline{1_{r(g^{-1}(r(g)|h))}\delta_{g^{-1}(r(g)|h)}} \in (T1_R)(1_RT) = T1_RT,
\end{align*}
for all $a \in A_{d(g)}$.
\end{proof}

Let $R$ be an idempotent ring (not necessarily unital). We say that a left $R$-module $M$ is \emph{unital} if $RM = M$. Moreover, $M$ is said to be \emph{torsion-free} when $Rm = 0$ implies $m = 0$ for any $m \in M$ (c.f. \cite{garcia1991morita}). The right module definitions are analogous. We will denote by $R\text{-mod}$ the category of unital left torsion-free $R$ modules (while $\text{mod-}R$ is defined in the same way). If $R$ is a unital ring, then $R\text{-mod}$ is the usual category of left $R$-modules.

According to \cite{garcia1991morita}, a \emph{Morita context}  is a sextuple $(R,R',M,M',\varphi,\varphi')$, where $R, R'$ are idempotent rings, $M$ is a $(R,R')$-bimodule, $M'$ is a $(R',R)$-bimodule and $\varphi : M \otimes_{R'} M' \to R$ and $\varphi' : M' \otimes_R M \to R'$ are bimodule morphisms such that:
\begin{enumerate}
    \item[(i)] $\varphi(x \otimes x')y = x\varphi'(x' \otimes y)$, for all $x,y \in M$ and $x' \in M'$,
    
    \item[(ii)] $\varphi'(x' \otimes x)y' = x'\varphi(x \otimes y')$, for all $x',y' \in M'$ and $x \in M$.
\end{enumerate}

By \cite[Proposition 2.6]{garcia1991morita}, if  $_RM$, $M'_R$, $M_{R'}$ and $_{R'}M'$ are unital modules and $\varphi$ and $\varphi'$ are surjective, then the categories $R$-mod and $R'$-mod (resp. mod-$R$ and mod-$R'$) are equivalent and the rings $R$ and $R'$ are said to be \emph{Morita equivalent}.

\begin{theorem} \label{contextodemorita}
The rings $R = A \ltimes_\alpha^o \G$ and $T = B \ltimes_\beta^o \G$ are Morita equivalent.
\end{theorem}
\begin{proof}
First note that $R$ is unital, since $\alpha$ is unital and $\G_0$ is finite. Hence, $R$ is idempotent. Since $B$ has local units and $\G_0$ is finite, we also have that $T$ has local units and therefore is idempotent.

Consider $M = 1_RT$ and $N = T1_R$. Clearly $M$ is a $(R,T)$-bimodule and $N$ is a $(T,R)$-bimodule. Defining the maps $\varphi : M \otimes_T N \to R$ and $\varphi' : N \otimes_R M \to T$ by $\varphi(m \otimes n) = mn$ and $\varphi'(n \otimes m) = nm$, we have that this maps are surjective and $(R,T,M,N,\varphi,\varphi')$ is a Morita context by Proposition \ref{propmorita}. It is also easy to see that $M_T,$ ${_RM}, N_R,$ and ${_TN}$ are unital modules since $R$ and $T$ are idempotent rings.
\end{proof}

\subsection{Partial Actions of Inverse Semigroups}

We will begin this subsection by reminding the reader of the notion of premorphism.

\begin{defi}
    \cite[p. 184]{gilbert2005actions} Let $\G$ and $\cH$ inductive groupoids. A map $\psi : \G \to \cH$ is said to be an \emph{inductive groupoid premorphism} if:
    \begin{enumerate}
        \item[(i)] $\psi(g)*\psi(h) \leq \psi(gh)$, for all $(g,h) \in \G_2$;

        \item[(ii)] $\psi(g)^{-1} = \psi(g^{-1})$, for all $g \in \G$;

        \item[(iii)] $\psi$ preserves orders.
    \end{enumerate}
\end{defi}

\begin{obs} \label{obspreforte} \begin{itemize}
    \item[1.] By \cite[Definition 2.6]{gould2011actions}, every inductive groupoid premorphism satisfies $\psi(g) * \psi(h) = \psi(gh) * d(\psi(h))$, for all $(g,h) \in \G_2$.
    \item[2.] By \cite[Lemma 4.2(b)]{gilbert2005actions}, we always have $d(\psi(g)) \leq \psi(d(g))$, by all $g \in \G$. 
    \item[3.] By \cite[Definition 2.6]{gould2011actions}, in all inductive groupoid premorphisms we have that if $e \in \G_0$ and $g \in \G$ are such that $e \leq d(g)$, then $d(\psi(g|e)) = \psi(e) \wedge d(\psi(g))$; and if $e \in \G_0$ and $g \in \G$ are such that $e \leq r(g)$, then $r(\psi(e|g)) = \psi(e) \wedge r(\psi(g))$.\end{itemize}
\end{obs}

Every inverse semigroup $S$ with set of idempotents $E(S)$ has a partial natural order $\preceq$ given by: \begin{center} $s \preceq_S t$ if and only if there exists $e \in E(S)$ such that $s = te$. \end{center}

\begin{defi} \label{defpremorfseminv}
    \cite[p. 19]{hollings2010extending} Let $S$ and $T$ be inverse semigroups. A map $\psi : S \to T$ is said to be an \emph{inverse semigroup premorphism} if, for all $s,t \in S$, it holds:
    \begin{enumerate}
        \item[(i)] $\psi(s)\psi(t) \preceq \psi(st)$;

        \item[(ii)] $\psi(s)^{-1} = \psi(s^{-1})$;

        \item[(iii)] If $s \preceq t$, then $\psi(s) \preceq \psi(t)$.
    \end{enumerate}
\end{defi}

The Ehresmann-Schein-Nambooripad Theorem establishes a relation between the categories of inductive groupoids and inverse semigroups. From now on, we will refer to it as the ESN Theorem, written below.

\begin{theorem}[ESN] \label{teoesn}
    \begin{enumerate}
        \item[(i)] \cite[Theorem 4.1.8]{lawson1998inverse} The category of inductive groupoids and inductive functors and the category of inverse semigroups and homomorphisms are isomorphic.

        \item[(ii)]  \cite[Theorem 6.1]{hollings2010extending} The category of inductive groupoids and inductive groupoid premorphisms and the category of inverse semigroups and inverse semigroup premorphisms are isomorphic.
    \end{enumerate}
\end{theorem}

Let $S$ be an inverse semigroup, and $E(S)$ the subset of idempotent elements of $S$. We adopt the definition of partial inverse semigroup action as stated in \cite[Definition 2.1]{beuter2019simplicity} with a slight weakening of the original assumptions. Here, we assume that $A_s$ is an ideal of $A_{ss^{-1}}$ instead of asking that $A_s$ is an ideal of $A$, for each $s \in S$.

\begin{defi}A partial action of an inverse semigroup $S$ on a ring $A$ is a collection $\alpha = (A_s,\alpha_s)_{s \in S}$ of ideals $A_{ss^{-1}}$ of $A$, ideals $A_s$ of $A_{ss^{-1}}$ and ring isomorphisms $\alpha_s : A_{s^{-1}} \to A_s$ such that 
\begin{enumerate}
    \item[(P1')] $A = \sum_{e \in E(S)} A_e$;

    \item[(P2')] $\alpha_s(A_{s^{-1}} \cap A_t) = A_s \cap A_{st}$, for all $s,t \in S$;

    \item[(P3')] $\alpha_s(\alpha_t(a)) = \alpha_{st}(a)$, for all $a \in A_{t^{-1}} \cap A_{(st)^{-1}}$, $s,t \in S$.
\end{enumerate} \end{defi}

The definitions of unital and preunital partial actions of inverse semigroups are analogous to the case of ordered groupoids. Next, we define the concept of globalization of a partial inverse semigroup action.

\begin{defi} \label{defglobseminv}
    Let $\alpha = (A_s,\alpha_s)_{s \in S}$ be a partial inverse semigroup action of an inverse semigroup $S$ on a ring $A$. A global action $\beta = (B_s,\beta_s)_{s \in S}$ of $S$ on a ring $B$ is said to be a \emph{globalization} of $\alpha$ if, for all $e \in E(S)$, there exists a ring monomorphism $\varphi_e : A_e \to B_e$ such that:
\begin{enumerate}
    \item[(i)] $\varphi_e(A_e)$ is an ideal of $B_e$;
    
    \item[(ii)] $\varphi_{ss^{-1}}(A_s) = \varphi_{ss^{-1}}(A_{ss^{-1}}) \cap \beta_s(\varphi_{s^{-1}s}(A_{s^{-1}s}))$;
    
    \item[(iii)] $\beta_s \circ \varphi_{s^{-1}s}(a) = \varphi_{ss^{-1}}  \circ \alpha_s(a)$, for all $a \in A_{s^{-1}}$;
    
    \item[(iv)] $B_s = \sum_{tt^{-1} = ss^{-1}} \beta_t\left ( \varphi_{t^{-1}t}(A_{t^{-1}t}) \right)$.
\end{enumerate}
\end{defi}

    % Observe that we already defined the equality in the index of the sum in (iv) in opposition to the inequality of Definition \ref{globalnotunique}(iv). This is due to the ESN Theorem, where we have that premorphisms (and, hence, partial actions) of inverse semigroups are related to premorphism (and, hence, strong partial actions) of inductive groupoids. Since every inductive groupoid is pseudoassociative by definition, this case naturally satisfies all the hypotheses of Theorem \ref{teounicglob} after the translation by the ESN Theorem. We will see more details about these statements and what we can conclude about existence and uniqueness of globalizations on the following.

We will denote by $\mathsf{Set}$ the category of sets and functions and by $\mathsf{Ring}$ the category of rings and ring homomorphisms. Given $A,B$ in $\text{Obj}(\mathsf{Ring})$ and $\varphi: A \to B$ in $\text{Mor}(\mathsf{Ring})$, we consider the \emph{forgetful functor} $F : \mathsf{Ring}  \to \mathsf{Set}$, where the image of the ring $A$ is the set $A$ (without the ring structure), and the image of the ring homomorphism $\varphi: A \to B$ is the function of sets $\varphi: A \to B$ (without the homomorphisms properties).

Let $A$ be a ring. We define $\mathcal{I}(A) := \mathcal{I}(F(A)) = \{ f : I \to J : I,J \subseteq A, f \text{ is a bijection}\}$. With the usual composition, we have that $\mathcal{I}(A)$ is an inverse semigroup. 

Let $S$ be an inverse semigroup and consider a partial action $\alpha = (A_s,\alpha_s)_{s \in S}$ of $S$ on a ring $A$. 

\begin{lemma}\label{alpha_induz}
 The partial action $\alpha$ gives rise to a premorphism $\gamma_{\alpha} : S \to \mathcal{I}(A)$, where $\gamma_{\alpha}(s)$ is a bijection between the subsets $\text{dom}(\gamma_\alpha(s)) = F(A_{s^{-1}})$ and $\text{Im}(\gamma_\alpha(s)) = F(A_s)$ of $F(A)$, defined as $\gamma_{\alpha}(s)(a) = \alpha_s(a)$, for all $a \in F(A_{s^{-1}})$.   
\end{lemma} 
\begin{proof} 
Condition (i) of premorphism follows from (P2') and (P3') of the definition of partial action. Condition (ii) of the premorphism follows from \cite[Proposition 2.2(c)]{beuter2019simplicity}. Condition (iii) of the premorphism follows from (PO).
\end{proof}

% By Lemma \ref{alpha_induz} and Theorem \ref{teounicglob}, we have the following result.

\begin{theorem}
Let $\alpha$ be a preunital partial action of an inverse semigroup $S$ on a ring $A$. Then $\alpha$ has a globalization $\beta$ if and only if $\alpha$ is unital. Moreover, $\beta$ is unique up to equivalences.
\end{theorem}
\begin{proof}
    Let $\gamma_\alpha$ be the inverse semigroup premorphism from $S$ to $\mathcal{I}(A)$ defined in Lemma \ref{alpha_induz}. Let $\mathbb{G}(\gamma_\alpha) : \mathbb{G}(S) \to \mathbb{G}(\mathcal{I}(A))$ be the inductive groupoid premorphism from the inductive groupoid $\mathbb{G}(S)$ to the inductive groupoid $\mathbb{G}(\mathcal{I}(A))$ associated to $\gamma_\alpha$ via the ESN Theorem. Since $\mathbb{G}(S)$ is inductive, it is pseudoassociative.

    Notice that the functor $\mathbb{G}$ of the ESN Theorem is such that $F(S) = F(\mathbb{G}(S))$ and $F(\gamma_\alpha) = F(\mathbb{G}(\gamma_\alpha))$ (considering the forgetful functor $F$ applied in the categories of inverse semigroups and of inductive groupoids, respectively). Hence, $\text{dom}(\mathbb{G}(\gamma_{\alpha})(s)) = F(A_{s^{-1}})$ and $\text{Im}(\mathbb{G}(\gamma_{\alpha})(s)) = F(A_s)$. Therefore, defining $\omega = (A_s, \omega_s)_{s \in \mathbb{G}(S)}$ by $\omega_s : A_{s^{-1}} \to A_s$, with $\omega_s(a) = \mathbb{G}(\gamma_\alpha)(s)(a) = \alpha_s(a)$, for all $a \in A_{s^{-1}}$, $s \in \mathbb{G}(S)$, we see that $\omega$ is a preunital strong P.O. action of $\mathbb{G}(S)$ on the ring $A$. Indeed, the conditions (P2) and (P3) are equivalent to the condition (i) of inductive groupoid premorphism. The condition (PO) corresponds to the condition (iii), while the condition (ii) of inductive groupoid premorphism is equivalent to \cite[Lemma 1.1]{bagio2012partial}(i). Finally, \cite[Definition 2.6]{gould2011actions} guarantees that (PS) holds, since all inductive groupoid premorphism $\psi$ satisfies $r(\psi(e|g)) = \psi(e) \wedge r(\psi(g))$.
    
    %and this is precisely what defines a strong P.O. action (and is equivalent to (PS) by Lemma \ref{condequivacparcialforte}), which does not occur in arbitrary P.O. actions, as we saw in Example \ref{exe2glob}.
    
   Now, we use Theorem \ref{teounicglob} to conclude that $\omega$ admits minimal globalization if and only if is unital. Since $\alpha$ is unital if and only if $\omega$ is unital, $\alpha$ admits globalization if and only if $\omega$ has a minimal globalization. In this case, we obtain a global action $\eta = (B_s,\eta_s)_{s \in \mathbb{G}(S)}$ of $\mathbb{G}(S)$ on a ring $B$ that is a minimal globalization of $\omega$. 
   
   Next, analogously to the construction of $\gamma_\alpha$, define the inductive groupoid premorphism $\rho_{\eta} : \mathbb{G}(S) \to \mathbb{G}(\mathcal{I}(A))$ such that the function $\rho_\eta(s) : F(B_{s^{-1}}) \to F(B_s)$ is given by $\rho_\eta(s)(b) = \eta_s(b)$, for all $b \in F(B_{s^{-1}})$. Notice that $\rho_\eta$ is indeed an inductive functor, because $\eta$ being global implies $B_{r(s)} = B_s$, for all $s \in \mathbb{G}(S)$, and consequently $\rho_\eta(s)\rho_\eta(t) = \rho_{\eta}(st)$, for all $(s,t) \in \mathbb{G}(S)_2$. 

       % Hence the conditions of ordered globalization of $\omega$ plus the minimal condition in $\omega$, and the definition of globalization of $\alpha$ are the exact same.
    Applying the inverse functor $\mathbb{S}$ of $\mathbb{G}$ via the ESN Theorem, we obtain an inverse semigroup homomorphism $\mathbb{S}(\rho_\eta) : S \to \mathcal{I}(A)$. Defining $\beta = (B_s,\beta_s)_{s \in S}$ by $\beta_s : B_{s^{-1}} \to B_s$ with $\beta_s(b) = \mathbb{S}(\rho_\eta)(s)(b) = \eta_s(b)$, for all $b \in B_{s^{-1}}$, $s \in S$, we have that $\eta$ being minimal globalization of $\omega$ implies that $\beta$ is the desired globalization of $\alpha$. The uniqueness (up to equivalence) of $\beta$ follows from Theorem \ref{teounicglob}.
    
    %since the constructed action $\eta$ is already minimal by the above argumentation.
\end{proof}

Consider the definition of partial skew inverse semigroup ring given in \cite{beuter2019simplicity}, that is, $A \ltimes_\alpha S = \mathcal{L}/\mathcal{N}$, where $\alpha$ is a unital partial action of $S$ on a ring $A$, $\mathcal{L} = \bigoplus_{s \in S} A_s\delta_s$ with component-wise addition and product given by $(a_s\delta_s)(b_t\delta_t) = \alpha_s(\alpha_{s^{-1}}(a_s)b_t)\delta_{st}$, and $\mathcal{N} = \langle a_s\delta_s - a_s\delta_t : s \leq t \text{ and } a_s \in A_s \rangle$ is an ideal of $\mathcal{L}$. Then the same argument used in Theorem \ref{contextodemorita} also proves the following result.

\begin{theorem}
    Let $\alpha$ be a unital partial inverse semigroup action of $S$ on a ring $A$. Assume that $E(S)$ is finite. Let $\beta$ be the minimal globalization of $\alpha$, which is a global action of $S$ on a ring $B$. Then $A \ltimes_\alpha S$ and $B \ltimes_\beta S$ are Morita equivalent.
\end{theorem}

\section*{Declarations}

\subsection*{Author's Contribution}
All authors wrote the main manuscript text and reviewed the manuscript.

\subsection*{Competing Interests}
 The authors have no competing interests as defined by Springer, or other interests that might be perceived to influence the results and/or discussion reported in this paper.

\subsection*{Ethical Approval}
Not applicable.

\subsection*{Availability of Data and Materials}
 Data sharing not applicable to this article as no datasets were generated or analysed during the current study.

\subsection*{Funding}
 There was no funding on the realization of this work.

\end{document}